\documentclass[11pt,twoside]{amsart}  
\usepackage{amssymb}  
  
\numberwithin{equation}{section}  
\newtheorem{theorem}{Theorem}[section]  
\newtheorem{lemma}[theorem]{Lemma}  
  
\newtheorem{corollary}[theorem]{Corollary}

\theoremstyle{definition}  
\newtheorem{definition}[theorem]{Definition}
\newtheorem{remark}[theorem]{Remark}  
\newtheorem{example}[theorem]{Example}

\newcommand{\A}{\mathcal{A}}

\newcommand{\F}{\mathbb{F}}
\newcommand{\R}{\mathbb{R}}
\newcommand{\Z}{\mathbb{Z}}

\begin{document}  
    
   
\title{Zero-nonzero patterns for nilpotent matrices over finite fields}

\author{Kevin N. Vander Meulen}  
\address{Department of Mathematics, Redeemer University College, Ancaster, ON
L9K 1J4 Canada}
\email{kvanderm@redeemer.ca}
\urladdr{http://cs.redeemer.ca/math/kvhome.htm}

\author{Adam Van Tuyl}  
\address{Department of Mathematical Sciences \\  
Lakehead University \\  
Thunder Bay, ON P7B 5E1, Canada}  
\email{avantuyl@lakeheadu.ca}  
\urladdr{http://flash.lakeheadu.ca/$\sim$avantuyl/}  
  
\thanks{Version: \today}  

\begin{abstract}
Fix a field $\F$.  A zero-nonzero pattern $\A$ is said to be potentially
nilpotent over $\F$ if there exists a matrix with entries in $\F$ with
zero-nonzero pattern $\A$ that allows nilpotence.  In this paper we initiate 
an investigation into which zero-nonzero patterns are potentially nilpotent
over $\F$, with a special emphasis on the
case that $\F = \Z_p$ is a finite field.  
As part of this investigation, we develop methods,
using the tools of algebraic geometry and commutative algebra, 
to eliminate zero-nonzero patterns $\A$
as being potentially nilpotent over any field $\F$.
We then use these techniques to classify all 
irreducible zero-nonzero patterns of order two
and three that are potentially nilpotent over $\Z_p$ for each prime $p$.
\end{abstract}
   
\maketitle  
  
  
\section{Introduction}  

A {\it zero-nonzero (znz) pattern} $\A$ is a square matrix whose entries come from
the set $\{*,0\}$ where $*$ denotes a nonzero entry.  Fix a field $\F$. 
We then set 
\[Q(\A,\F) = \{ A \in M_n(\F) ~|~ (A)_{i,j} \neq 0 \Leftrightarrow 
(\A)_{i,j} = * ~~\mbox{for all $i,j$}\}.\]
The set $Q(\A,\F)$, sometimes denoted $Q(\A)$ when $\F$ is known, is usually called the {\it qualitative class} of $\A$.
An element $A \in Q(\A,\F)$
is called a {\it matrix realization} of $\A$.  
 A znz-pattern $\A$ is said to be {\it potentially
nilpotent over} $\F$ if there exists a matrix $A \in Q(\A,\F)$ such that 
$A^k = 0$ for some positive integer $k$.  
In this paper we study the question of what patterns $\A$ are potentially
nilpotent over a field $\F$.  Although we will present some results for arbitrary fields,
we are particularly interested in the case that $\F$ is a finite field.

One motivation to study this question is to provide a first step
in understanding spectrally arbitrary patterns 
in the context of fields other than $\R$.  An $n \times n$
znz-pattern $\A$ is a {\it spectrally
arbitrary pattern} (SAP) if given any monic polynomial $p(x)$ of degree $n$
with coefficients in $\F$, there exists
a matrix $A \in Q(\A,\F)$ whose characteristic polynomial is $p(x)$.
There is a growing body of literature (see, for example, \cite{Britz, CM,DHHMMPSV, P, BS} and their references)
interested in identifying patterns that are SAPs
when $\F = \R$ (with much focus on \emph{sign} patterns: patterns whose entries come from the set $\{ +,-,0\}$).  
However, little work has been done on this question when $\F$ is field different from $\R$.
Since any SAP is automatically potentially nilpotent, 
the topic of this paper can be seen as a step in identifying patterns which could be SAPs.
Additional work on 
the problem of SAPs over finite fields
is under development by E. Bodine~\cite{B}. 

We now cursorily survery the problem of identifying potentially nilpotent patterns over $\F = \R$.
Determining when a sign pattern is potentially nilpotent was listed
as an open problem in \cite{EJ}.
Potentially nilpotent star sign patterns were introduced
in \cite{Y} and fully characterized in \cite{MTvdD}.  
Potentially nilpotent sign patterns of order up to 3 were
characterized by \cite{EL}.  Included in \cite{EL} is an investigation of
sign patterns that allow nilpotence of index $2$, where the index
of matrix $A$ is the smallest integer $k$ such that $A^k = 0$;  this was later extended
to index $3$ in \cite{GLS} (see also \cite{Br}).  
In \cite{P}, it was shown that all potentially nilpotent full sign
patterns (i.e. patterns with no zero entries) are also 
SAPs.  Consequently, recent work \cite{KOvdDvdHVM} presents constructions
of potentially nilpotent full sign patterns.  
Much work in determining when a pattern
is potentially nilpotent occurs in the literature on SAPs.
Identifying potentially nilpotent patterns over $\R$ is
in part an important subproblem in the study of SAPs due to
a technique developed in \cite{DJOvdD}, usually referred to as the Nilpotent-Jacobi Method.
Roughly speaking, if $A$ is a nilpotent realization over $\R$ of a pattern $\A$,
then one can determine if $\A$ is spectrally arbitrary by evaluating the
entries of $A$ in a Jacobian matrix constructed from $\A$.  Note that this technique
requires the Implicit Function Theorem, which holds over $\R$, so one
should not expect a generalization of this
approach to arbitrary fields.

We begin in Section~\ref{basic} by reviewing some basic
results concerning nilpotent matrices over a field $\F$.  Many of the results
that are known to hold in $\R$ continue to hold over an arbitrary field.

In  Section~\ref{ideal} we introduce some techniques to eliminate certain
patterns as being potentially nilpotent over a field.  
We use some tools from commutative
algebra and algebraic geometry to carry out this program.  
Starting with a znz-pattern $\A$ with nonzero entries
at $(i_1,j_1),\ldots,(i_t,j_t)$, we define an ideal $I_{\A}$ in a 
polynomial ring $R_{\A} = \F[z_{i_1,j_1},\ldots,z_{i_t,j_t}]$ over the field $\F$.
In Theorem \ref{characterization} we show that  $\A$ is potentially
nilpotent over a field $\F$ if and only if a certain subset 
of the affine variety defined by $I_{\A}$ is nonempty.  With this characterization,
we can use the technique of ideal saturation (see Definition \ref{satdef})
to determine if a given pattern is {\it not}  potentially nilpotent:

\noindent
{\bf Theorem \ref{saturation}.} {\it
Let $\F$ be any field and
$\A$ a znz-pattern.
Let $J = (z_{i_1,j_1}\cdots z_{i_t,j_t})$ be the ideal generated by
the product of the variables of $R_{\A}$.  If $1 \in I_{\A}:J^{\infty}$,
then $\A$ is not potentially nilpotent over any extension of $\F$. 
}

Since many computer algebra programs can compute the saturation of ideals, Theorem \ref{saturation}
promises to be a useful tool for future experimentation.  In the last part of Section~\ref{ideal}
we review the basics of Gr\"obner bases, and show how Gr\"obner bases can also
be used to eliminate znz-patterns as being potentially nilpotent (see Example \ref{converse}).

As an aside, we hope that our results, along with the 
work of Shader \cite{BS} and Kaphle \cite{K}, will highlight
the usefulness of techniques from commutative algebra and algebraic geometry
in the study of SAPs.  Shader uses a result about the number of algebraically independent
elements over the polynomial ring $\R[x_1,\ldots,x_n]$ to prove a lower bound
on the number of nonzero entries in a SAP.  Kaphle's MSc thesis uses Gr\"obner bases
to eliminate sign patterns as being potentially nilpotent.  Note that one difference between
our work and the work of Kaphle is that we use the 
equations constructed from the characteristic polynomials
when forming the Gr\"obner basis, while Kaphle uses equations constructed from
the traces of the matrices $\A^k$ for $k=1,\ldots,n$.

In Section~\ref{condition} we introduce a necessary condition for a znz-pattern
$\A$ to be nilpotent over a field $\F$.  Precisely,
we look at znz-patterns $\A$ where $\A$ is irreducible and the digraph
$D(\A)$ has no $2$-cycles (see Section~\ref{basic}).  When $\F = \R$, if $\A$
has at least two nonzero entries on the diagonal and
$\A$ is potentially nilpotent, then it is known (see \cite{DHHMMPSV}) that $D(\A)$ has to have
a $2$-cycle.  However, we show that this is no longer true over an arbitrary field.
What is important is that the polynomial $x^3-1$ factors completely
over $\F$.  In fact, we prove a more general result:

\noindent
{\bf Theorem \ref{nomcycle}.} {\it 
Let $\A$ be a znz-pattern with $m \geq 2$ nonzero entries
on the diagonal, and suppose that $D(\A)$ has no $k$-cycles
with $2 \leq k \leq m-1$.  If $\A$ is potentially nilpotent
over $\F$, then the polynomial $x^m-1$ factors into 
$m$ linear forms over $\F$.}

Our paper culminates with Section~\ref{small} which
uses the above techniques to classify all potentially nilpotent patterns of order at most three 
when $\F = \Z_p$ is the finite field with $p$ elements, where $p$ is a prime
(see Theorems \ref{2x2case} and \ref{3x3case}).  One 
interesting by-product of this classification is the discovery
that $\A$ may be potentially
nilpotent in a field $\F$, but a {\it superpattern} of $\A$, that is, a znz-pattern
$\A'$ such that $(\A')_{i,j} \neq 0$ whenever $(\A)_{i,j} \neq 0$, may
not be potentially nilpotent over the same field $\F$.

\noindent
{\bf Acknowledgments.}  Many of these results were inspired through
computer experiments using CoCoA \cite{C} and Macaulay 2 \cite{M2}.  The second author
thanks Redeemer University College for its hospitality while working
on this project, while both authors thank Natalie Campbell for helpful discussions.  Both authors 
were supported by an NSERC grant.


\section{Basic Properties}\label{basic}

In this section, we review some of the needed properties
of znz-pattern matrices and summarize some of the basic properties of potentially
nilpotent matrices over $\F$.  Some of these results were known when $\F = \R$;  we consider the 
more general case.  We continue to use the notation from the introduction.

When referring to elements of the field $\F$, we shall use $1_\F$ to denote
the multiplicative identity of $\F$, but
abuse notation slightly and write $0$ for the additive identity $0_\F$.  
For any positive integer $n \in \Z$,
we write $n_\F$ to denote $(1_\F + \cdots + 1_{\F})$ ($n$ times).  
Then $-n_\F$ will denote the additive inverse of $n_\F$ in $\F$,
and $n^{-1}_{\F}$ denotes the multiplicative inverse (provided $n_\F \neq 0$).

Given an $n \times n$ znz-pattern $\A$, we can construct a digraph
$D(\A) = (V,E)$ on the vertex set $V = [n]:= \{1,\ldots,n\}$,
whose edge set consists of the arcs $(i,j)$ whenever $(\A)_{i,j} \neq 0$.
We call the edge $(i,i)$ a loop;  loops correspond to the nonzero
entries on the diagonal of $\A$.  A simple cycle $\gamma$ of length $k$,
also called a $k$-cycle, is a sequence of $k$ distinct vertices
$\{i_1,\ldots,i_k\}$ such that $(i_1,i_2),(i_2,i_3),\ldots,(i_{k-1},i_k),(i_k,i_1) \in E$.  We sometimes
denote a $k$-cycle $\gamma$ by $(i_1,i_2,\ldots,i_k)$, and denote its length
by $|\gamma| = k$.  Furthermore, we say two cycles $\gamma_1$ and $\gamma_2$
are disjoint if they have no vertices in common.

Suppose that $A \in Q(\A,\F)$ is a realization of $\A$.  The characteristic
polynomial of $A$ can be described in terms of the cycles of $D(\A)$.  
Precisely, suppose that $\gamma = (i_1,\ldots,i_k)$ is a $k$-cycle.
We let $\prod(\gamma) = a_{i_1,i_2}a_{i_2,i_3}\cdots a_{i_k,i_1}$ where
$a_{i,j} = (A)_{i,j}$.  Then the characteristic polynomial of $A$ has the form
\[p_A(x) = x^n + r_1x^{n-1} + r_2x^{n-2} + \cdots + r_{n-1}x + r_n\]
where
\[r_i = (-1)^{i}\sum_{\begin{array}{c}
\gamma_1,\ldots,\gamma_p ~~\mbox{pairwise disjoint cycles} \\
|\gamma_1|+\cdots + |\gamma_p| = i
\end{array}}
\left[(-1)^{|\gamma_1|-1}\prod(\gamma_1) \cdots (-1)^{|\gamma_p|-1}\prod(\gamma_p)\right].\]

A znz-pattern $\A$ of order $n \geq 2$ is {\it reducible} if there exists some integer
$1 \leq r \leq n-1$ and a permutation matrix $P$ such that 
\[P\A P^T =
\begin{bmatrix}
\A_1 & \A_2 \\
0_{r,n-r} & \A_3 
\end{bmatrix}. \]
Otherwise, a znz-pattern $\A$ is called {\it irreducible}.  Equivalently,
a znz-pattern $\A$ is irreducible if and only if the associated digraph $D(\A)$ is
strongly connected, that is, there is a directed path between
any pair of distinct vertices.   The Frobenius normal form of $\A$ is a permutation similar block
upper triangular matrix whose diagonal blocks are irreducible.  The diagonal blocks
are called the irreducible components of $\A$.  

The final lemma of this section summarizes some of the results
we will need in the later sections.

\begin{lemma}  \label{irreduciblereduction}
Fix a field $\F$ and a znz-pattern $\A$.
\begin{enumerate}
\item[(a)]  Suppose that
$\A$ is reducible with irreducible components $\A_1,\ldots,\A_t$.  Then
$\A$ is potentially nilpotent over $\F$ if and only if each znz-pattern
$\A_i$ is potentially nilpotent.
\item[(b)] If $\A$ is potentially nilpotent over
$\F$, then so is $\A^{\mbox{T}}$, the transpose of $\A$.
\item[(c)] If $A$ is a nilpotent realization of $\A$, then
the characteristic polynomial of $A$ is $p_A(x) = x^n$.
\end{enumerate}
\end{lemma}


\section{Eliminating potentially nilpotent candidates via ideal saturation}\label{ideal}

Let $\F$ be any field.  Using some
tools and techniques from commutative algebra and algebraic geometry,
we will show that a znz-pattern $\A$ is potentially nilpotent over $\F$
if and only if a certain geometric set is nonempty.  As an application,
we develop an algebraic 
method to eliminate certain znz-patterns $\A$ as being
potentially nilpotent over $\F$.  We also explain how to use
Gr\"obner bases to show that some patterns are not potentially nilpotent.
While we will endeavor to keep this
material as self-contained as possible, further background material can be
found in the book
of Cox, Little, and O'Shea \cite{CLO}.

We begin with some notation.  Fix a znz-pattern $\A$,
and let $S_{\A} = \{(i,j) ~|~ (\A)_{i,j} \neq 0\}$ be 
the locations of the nonzero elements in $\A$.  We then define the polynomial ring
\[R_{\A} := \F[z_{i,j} ~|~ (i,j) \in S_{\A}] = \F[z_{i_1,j_1},\ldots,z_{i_t,j_t}]\] 
in $t = |S_{\A}|$ variables over the field $\F$.  
Associate to $\A$ the matrix $M_{\A}$ where 
\[(M_{\A})_{i,j} := \left\{
\begin{array}{ll}
z_{i,j} & \mbox{if $(\A)_{i,j} \neq 0$} \\
0 & \mbox{if $(\A)_{i,j} = 0$.}
\end{array}\right.\]
Note that $M_{\A}$ is not a realization of $\A$ since the entries of $M_{\A}$ are 
variables.  
The characteristic polynomial of $M_{\A}$ then has the form
\[p_{M_{\A}}(x) = x^n - F_1x^{n-1} + F_2x^{n-2} + \cdots + (-1)^{n-1}F_{n-1}x + (-1)^nF_n\]
where each coefficient $F_i = F_i(z_{i_1,j_1},\ldots,z_{i_t,j_t})$ is a polynomial in $R_{\A}$.
We then use the $n$ coefficients of the characteristic polynomial
to define an ideal of $R_{\A}$.  Precisely, let 
\[
I_{\A} := (F_1,\ldots,F_n) \subseteq R_{\A}.
\]
In fact, $I_{\A}$ is a homogeneous ideal since for each $F_i \neq 0$,
the polynomial $F_i$ is a homogeneous polynomial of degree $i$;  
recall that a polynomial $G$ is {\it homogeneous} if each
term of $G$ has the same degree.  To see this fact, note that
each term of $F_i$ corresponds
to a composite cycle of length $i$ in the directed graph $D(\A)$ (see the formula
in Section~\ref{basic}), from which it follows that $F_i$ is homogeneous.
Hence, every znz-pattern $\A$ induces a homogeneous ideal $I_{\A}$.

\begin{example}\label{examplenotation}
We illustrate the above notation with the following znz-pattern
\[\A = \begin{bmatrix} * & * & 0 \\
* & 0 & * \\
0 & * & *
\end{bmatrix}.\]
The associated matrix is
\[M_{\A} = \begin{bmatrix} z_{1,1} & z_{1,2} & 0 \\
z_{2,1} & 0 & z_{2,3} \\
0 & z_{3,2} & z_{3,3}
\end{bmatrix}\]
where the $z_{i,j}$'s are indeterminates in the 
polynomial ring $R_{\A} = \F[z_{1,1},z_{1,2},z_{2,1},z_{2,3},z_{3,2},z_{3,3}]$.
The ideal $I_{\A}$ is then generated by three homogeneous polynomials:
\[I_{\A} =  (z_{1,1} + z_{3,3}, \hspace{.2cm} z_{1,2}z_{2,1} + z_{2,3}z_{3,2} + z_{1,1}z_{3,3}, 
\hspace{.2cm}z_{1,1}z_{2,3}z_{3,2} + z_{1,2}z_{2,1}z_{3,3}).\]
\end{example}

 For each $\underline{a} = (a_{i_1,j_1},\ldots,a_{i_t,j_t}) \in \F^{t}$, 
let $M_{\A}(\underline{a})$
denote the matrix obtained by replacing each $z_{i_k,j_k}$ with $a_{i_k,j_k}$.
The characteristic polynomial of 
$M_{\A}(\underline{a})$ will have the form:
\[p_{M_{\A}(\underline{a})}(x) =x^n - F_1(\underline{a})x^{n-1} + F_2(\underline{a})x^{n-2} + \cdots + (-1)^{n-1}F_{n-1}(\underline{a})x + (-1)^nF_n(\underline{a}).\]
If $\A$ is potentially nilpotent over $\F$, 
then there exists an $\underline{a} \in \F^{t}$ with all $a_{i,j} \neq 0$ such that 
$M_{\A}(\underline{a})$  is a nilpotent matrix.  In particular, 
the characteristic polynomial of $M_{\A}(\underline{a})$ must
be $x^n$ by Lemma \ref{irreduciblereduction} (c), which, in turn, implies that 
$F_i(\underline{a}) = 0$ for $i = 1,\ldots,n$.  Thus, one can determine if
a znz-pattern $\A$ is potentially nilpotent over $\F$ if one 
understands the affine variety described by $I_{\A}$;
the {\it affine variety}\footnote[1]{What we call 
an affine variety is sometimes called an algebraic set.  We have
decided to follow \cite{CLO}.} defined by  $I_{\A}$
is the set 
\begin{eqnarray*}
V(I_{\A}) &=& \{\underline{a} \in \F^t ~|~ G(\underline{a}) = 0
~~\mbox{for all $G \in I_{\A}$}\} \\
& = & \{\underline{a}\in \F^t ~|~ F_1(\underline{a}) = \cdots = F_n(\underline{a}) = 0 \}.
\end{eqnarray*}
The set $V(I_{\A})$ contains all the elements $\underline{a} \in \F^t$ such that
the matrix $M_{\A}(\underline{a})$ is nilpotent.
Thus, if $\A$ is potentially nilpotent over $\F$ and $M_{\A}(\underline{a})$ is a 
realization of $\A$
that is nilpotent,
then $\underline{a} \in V(I_{\A})$.  However, the converse is not necessarily true.
Indeed, if $\underline{b} \in V(I_{\A})$, then while $M_{\A}(\underline{b})$ still has
a characteristic polynomial of $x^n$, the matrix $M_{\A}(\underline{b})$ may not be 
a realization  of $\A$.  As a simple example, note that $\underline{0} = (0,\ldots,0) \in V(I_{\A})$,
(since each $F_i$ is homogeneous, and thus $F_i(\underline{0}) = 0$
for all $i$), but $M_{\A}(\underline{0})$ 
is the zero-matrix,
which is not a realization of  $\A$.

For each indeterminate $z_{i,j} \in R_{\A}$, let $V(z_{i,j})$ denote the associated
affine variety, that is, $V(z_{i,j}) = \{\underline{a} \in \F^t ~|~ a_{i,j} = 0 \}$.
With this notation, we can determine if a znz-pattern $\A$ is potentially nilpotent over
$\F$:

\begin{theorem}  \label{characterization}
Fix a field $\F$ and a znz-pattern $\A$.  Then
$\A$ is potentially nilpotent over $\F$ if and only if 
\[V(I_{\A}) \setminus \bigcup_{k=1}^{t} V(z_{i_k,j_k}) \neq \emptyset.\]
\end{theorem}

\begin{proof}
If $\A$ is potentially nilpotent over $\F$, then there exists an $\underline{a} \in \F^t$
such that $M_{\A}(\underline{a})$ is nilpotent.  But that implies that
$\underline{a} \in V(I_{\A})$.  Furthermore, since $M_{\A}(\underline{a})$
is a realization of $\A$, $a_{i_k,j_k} \neq 0$ for $k =1,\ldots,t$, or in other
words, $\underline{a} \not\in   V(z_{i_k,j_k})$ for each $k$.  This proves the
first direction.  

For the reverse direction, if $\underline{a} \in V(I_{\A}) \setminus \bigcup_{k=1}^{t} V(z_{i_k,j_k})$,
then $M_{\A}(\underline{a})$ is a nilpotent matrix, and furthermore, since
$a_{i_k,j_k} \neq 0$ for all $k$, this matrix is also a realization of $\A$.
\end{proof}

As a consequence of Theorem \ref{characterization},
to determine if $\A$ is potentially
nilpotent over $\F$, it suffices to show that the set $V(I_{\A}) \setminus \bigcup_{k=1}^{t} V(z_{i_k,j_k})$
is non-empty.  Unfortunately,  this can be a highly non-trivial problem.  However
we can use
this reformulation to describe an algebraic method to determine if
the set  $V(I_{\A}) \setminus \bigcup_{k=1}^{t} V(z_{i_k,j_k})$ is empty, thus providing
a means to determine if 
$\A$ is not potentially nilpotent over $\F$.

We begin with a simple lemma.  A {\it monomial} of $R_{\A}$ is any polynomial of 
the form $m = z_{i_1,j_1}^{b_1}z_{i_2,j_2}^{b_2}\cdots z_{i_t,j_t}^{b_t}$ with each 
$b_i \in \mathbb{Z}_{\geq 0}$.

\begin{lemma}  \label{monomiallemma}
Fix a field $\F$ and a znz-pattern $\A$.  Suppose that
there exists a monomial $m = z_{i_1,j_1}^{b_1}z_{i_2,j_2}^{b_2}\cdots z_{i_t,j_t}^{b_t}
\in I_{\A}$.  Then $\A$ is not potentially nilpotent over $\F$.
\end{lemma}

\begin{proof}
For any $\underline{a} \in V(I_{\A})$, we must have
$m(\underline{a}) = a_{i_1,j_1}^{b_1} \cdots a_{i_t,j_t}^{b_t} = 0$
because $m \in I_\A$.  But this means that $a_{i_k,j_k}  =0$ for
some $k = 1,\ldots,n$, and thus, $\underline{a} \in V(z_{i_k,j_k})$.  Now
apply Theorem \ref{characterization}.
\end{proof}

The colon operation and the saturation of ideals are two required algebraic ingredients:

\begin{definition}  \label{satdef}
Let $I$ and $J$ be ideals of a ring $R$.  Then
$I:J$ denotes the ideal
\[I:J = \{g \in R ~|~ gJ \subseteq I \}.\]
The {\it saturation
of $I$ with respect to $J$}, denoted $I:J^{\infty}$, is the ideal
\[I:J^{\infty} = \{g \in R ~|~ gJ^i \subseteq I ~~\mbox{for some integer $i \geq 0$}\}.\]
Alternatively, $I:J^{\infty} = (\cdots (((I:J):J):J) \cdots).$
\end{definition}

We come to one of the main results of this section.

\begin{theorem}  \label{saturation}
Fix a field $\F$ and a znz-pattern $\A$.  If  
$R_{\A} = \F[z_{i_1,j_1},\ldots,z_{i_t,j_t}]$,
then  let $m_{\A} := \prod_{k=1}^t z_{i_k,j_k}$ and let $J = (m_{\A})$
be the ideal generated by $m_{\A}$. 
Then
\begin{enumerate}
\item[(a)]  $V(I_{\A}) \setminus \bigcup_{k=1}^{t} V(z_{i_k,j_k}) \subseteq V(I_{\A}:J^{\infty})
\subseteq V(I_{\A}:J)$;
\item[(b)]  if $1 \in I_{\A}:J^{\infty}$, then $\A$ is not potentially nilpotent
over $\F$, or any field extension of $\F$;
\item[(c)]  if $1 \in I_{A}:J$, then $\A$ is not potentially nilpotent
over $\F$, or any field extension of $\F$.
\end{enumerate}
\end{theorem}

\begin{proof}
Statement $(a)$ is a well-known result via the algebraic geometry dictionary.  For completeness,
we include a short proof in this context.  Suppose that $\underline{a} \in V(I_{\A}) \setminus \bigcup_{k=1}^{t} V(z_{i_k,j_k})$, and thus, $a_{i_k,j_k} \neq 0$ for $k=1,\ldots,t$.  Suppose that
$G \in I_{\A}:J^{\infty}$.  Thus, there exists an integer $i$ such that 
$GJ^i \subseteq I_{\A}$.  But because $J^i = (m_{\A}^i)$, this implies that
$Gm_{\A}^i \in I_{\A}$.  Since $\underline{a} \in V(I_{\A})$, we have $(Gm^i_{\A})(\underline{a})
= G(\underline{a})m^i_{\A}(\underline{a}) = 0$.  But since each $a_{i_k,j_k} \neq 0$,
we have $m^i_{\A}(\underline{a}) = a_{i_1,j_1}^i\cdots a_{i_t,j_t}^i \neq 0$, and hence $G(\underline{a}) = 0$,
or equivalently, $\underline{a} \in  V(I_{\A}:J^{\infty})$.  The second inclusion
containment follows from the fact that $I_{\A}:J \subseteq I_{\A}:J^{\infty}$.

To prove $(b)$, suppose that $1 \in I_{\A}:J^{\infty}$.  It then follows that there exists
an $i$ such that $J^i \subseteq I_{\A}$, and hence $m_{\A}^i \in I_{\A}$.  But then
we get the desired conclusion by Lemma \ref{monomiallemma}.
In any extension of $\F$, we will
continue to have $m_{\A}^i \in I_{\A}$.  Statement $(c)$ follows directly
from $(b)$ since we will have $1 \in I_{\A}:J \subseteq I_{\A}:J^{\infty}$.  
\end{proof}

\begin{remark} Many computer algebra systems allow one to compute the saturation
of an ideal, thus making Theorem \ref{saturation} a practical tool.  The
computational commutative algebra programs CoCoA \cite{C} and Macaulay 2 \cite{M2}
are two free programs that can be used to compute 
the ideals $I:J$ and $I:J^{\infty}$.  On the 
second author's web page\footnote[2]{{\tt http://flash.lakeheadu.ca/$\sim$avantuyl/research/research.html}}
is a short introduction on how to use these programs to compute the examples
found below.
\end{remark}

Some well-known necessary facts for nilpotent matrices are simple corollaries of Theorem
\ref{saturation}.

\begin{corollary}\label{corsaturation}
Let $\A$ be znz-pattern.  If $\A$ has only one nonzero entry on the diagonal or
only one transversal, then $\A$ is not potentially nilpotent over any field $\F$.
\end{corollary}

\begin{proof} In both cases, we show that one of the generators of $I_{\A}$ must
be a monomial.

 If $\A$ has only one nonzero entry on the diagonal, say at position
$(i,i)$, then the trace of $M_{\A}$ is $z_{i,i}$.  But since $F_1 = \operatorname{tr} M_{\A} = z_{i,i}$, 
it immediately follows that $m_{\A} \in I_{\A}$, and hence,  $1 \in I_{\A}:(m_{\A})$.
Similarly, if $\A$ has only one transversal, the determinant of $M_{\A}$, which equals $F_n$,
has form $z_{i_1,j_1}^{b_1}z_{i_2,j_2}^{b_2}\cdots z_{i_t,j_t}^{b_t}$ where
$b_k = 1$ or $0$.   It then follows that $m_{\A} \in I_\A$, or equivalently,
$1 \in I_{\A}:(m_{\A})$.
\end{proof}

We now provide some illustrative examples.

\begin{example}\label{examplesaturation}
Let $\A$ be the znz-pattern of Example \ref{examplenotation}.
Let $\F$ be any field of characteristic two.
Because 
\[I_{\A} =  (z_{1,1} + z_{3,3}, ~~z_{1,2}z_{2,1} + z_{2,3}z_{3,2} + z_{1,1}z_{3,3}, 
~~z_{1,1}z_{2,3}z_{3,2} + z_{1,2}z_{2,1}z_{3,3}),\]
the monomial $z^2_{1,1}z_{3,3} \in I_\A$ because
\[z_{1,2}z_{2,1}(z_{1,1}+z_{3,3}) + z_{1,1}(z_{1,2}z_{2,1} + z_{2,3}z_{3,2} + z_{1,1}z_{3,3}) +  (z_{1,1}z_{2,3}z_{3,2} + z_{1,2}z_{2,1}z_{3,3})\]
\[=  2z_{1,2}z_{2,1}z_{1,1} + 2z_{1,1}z_{2,3}z_{3,2} + z^2_{1,1}z_{3,3}  + 2z_{1,2}z_{2,1}z_{3,3} = z^2_{1,1}z_{3,3}\]
since $x+x = 0$ for any $x \in \F$.
Thus $\A$ is not potentially nilpotent over any field of extension of $\F$.
Note that when $\F = \Z_2$ is the finite field with exactly two elements,
then one could use a direct calculation because there is only one choice for each $z_{i,j}$, 
namely $1_{\F}$.
However, this method shows that $\A$ is not potentially nilpotent over any
extension of this field.
\end{example}

\begin{example}
It is possible that $1 \in I_{\A}:J^{\infty}$, but $1 \not\in I_{\A}:J$.  As an example,
consider the znz-pattern
\[\A =  \begin{bmatrix} * & 0 & 0 \\
0 & * & * \\
0 & * & *
\end{bmatrix}.\]
We can see immediately that $\A$ is not potentially nilpotent over any field $\F$
since any realization $A$ of $\A$ will have the nonzero eigenvalue of $a_{1,1}$.
However, this cannot be deduced from $I_{\A}:J$.  For example,
if $\F = \Z_2$, then we use CoCoA or Macaulay 2 to find 
\footnotesize
\begin{eqnarray*}
I_{\A}:J &=& (z_{1,1} + z_{2,2} + z_{3,3}, -z_{1,1}z_{2,2} + z_{2,3}z_{3,2} - z_{1,1}z_{3,3} - z_{2,2}z_{3,3}, 
-z_{1,1}z_{2,3}z_{3,2} + z_{1,1}z_{2,2}z_{3,3}):(m_{\A}) \\
& = & (z_{1,1} + z_{2,2} + z_{3,3}, z_{2,2}^2 + z_{3,3}^2, z_{2,3}z_{3,2} + z_{2,2}z_{3,3}).
\end{eqnarray*}
\normalsize
However, a computer algebra system will reveal that $1 \in I_{\A}:J^{\infty}$, thus showing that $\A$ is not potentially nilpotent
over $\F$.
\end{example}

\begin{example}  Using the saturation of ideals also lends itself
to sign patterns.  Consider the signed pattern
\[\A 
=
\begin{bmatrix} - & - & - & 0 & 0 \\
+ & + & + & 0 & 0 \\
0 & 0 & 0 & - & - \\
0 & - & 0 & 0 & - \\
- & 0 & 0 & 0 & 0 
\end{bmatrix}.\]
The pattern $\A$ is the pattern $\mathcal{G}_5$ studied in \cite{KOvdD}.  We then consider
the matrix
\[M_{\A} 
=
\begin{bmatrix} -z_{1,1} & -z_{1,2} & -z_{1,3} & 0 & 0 \\
z_{2,1} & z_{2,2} & z_{2,3} & 0 & 0 \\
0 & 0 & 0 & -z_{3,4} & -z_{3,5} \\
0 & -z_{4,2} & 0 & 0 & -z_{4,5} \\
-z_{5,1} & 0 & 0 & 0 & 0 
\end{bmatrix}.\]
We define $I_{\A}$ as above.  Letting $\F = \R$, we find that 
$1 \in I_{\A}:(m_{\A})^{\infty}$ using CoCoA.  This sign pattern $\A$, therefore,
is not potentially nilpotent over $\R$, as first discovered in \cite{KOvdD};
in fact $\mathcal{G}_5$ is part of a much larger family of non-potentially
nilpotent patterns.
\end{example}

As we will show below, the converse of  Theorem \ref{saturation} (b) does not hold.
To show that $\A$ is not potentially nilpotent, we apply
the theory of Gr\"obner bases.  Roughly
speaking, a Gr\"obner basis of $I_{\A}$ is a ``good'' choice of generators of $I_{\A}$
which can allow one to describe the affine variety $V(I_{\A})$.   

We now recall the needed definitions.
We fix a {\it monomial ordering} $>$ on the monomials of $R_{\A}$, that is, 
(1) $>$ is a total ordering on the set of monomials,
(2) $>$ is compatible with multiplication 
(if $m_1 > m_2$, then for any monomial $m$, $mm_1 > mm_2$),
and (3) $>$ is also a well-ordering.  
Of particular importance is the {\it lex monomial ordering}, 
that is, 
\[z_{i_1,j_1}^{a_1}z_{i_2,j_2}^{a_2}\cdots z_{i_t,j_t}^{a_t} >z_{i_1,j_1}^{b_1}z_{i_2,j_2}^{b_2}\cdots z_{i_t,j_t}^{b_t}\] 
if and only if the first nonzero entry of the $t$-tuple
$(a_1-b_1,\ldots,a_t-b_t)$ is positive.

For any polynomial $F = \sum c_{\alpha}m_{\alpha} \in R_{\A}$
where $m_{\alpha}$ are monomials and $c_{\alpha} \in \F$, 
the leading term of $F$, denoted $\operatorname{LT}_>(F)$ is the largest monomial term
$c_{\alpha}m_\alpha$ in $F$ with respect to $>$.  

\begin{definition}
A subset $\{G_1,\ldots,G_s\}$ of an ideal
$I$ is a {\it Gr\"obner  basis} of $I$ with respect to a monomial
ordering $>$ if for all $F \in I$, $LT_>(F)$
is divisible by $LT_>(G_i)$ for some $i$. 
\end{definition}

We then make use of the following two properties of Gr\"obner bases.

\begin{theorem}  \label{grobnerprop}
Let $R = \F[z_{i_1,j_1},\ldots,z_{i_t,j_t}]$.  Let
$>$ be the lex monomial ordering with the property that $z_{i_1,j_1} > \cdots > z_{i_t,j_t}$.
Let $I$ be an ideal of $R$, and suppose that $\{G_1,\ldots,G_s\}$
is a Gr\"obner basis of $I$ with respect to $>$.  Then
\begin{enumerate}
\item[(a)] $I = (G_1,\ldots,G_s)$, that is, the Gr\"obner basis generates $I$;
\item[(b)] Let $I_l = I \cap \F[z_{i_{l+1},j_{l+1}},\ldots,z_{i_t,j_t}]$.  Then
$I_l$ is the $l$th elimination ideal, and a Gr\"obner basis for 
$I_l$ is $\{G_1,\ldots,G_s\} \cap \F[z_{i_{l+1},j_{l+1}},\ldots,z_{i_t,j_t}]$.
\end{enumerate} 
\end{theorem}

\begin{proof} Statement (a) is \cite[Chapter 2, $\S$5, Corollary 2]{CLO}, 
while (b) is known as the
Elimination Theorem \cite[Chapter 3, $\S$1, Theorem 2]{CLO}. 
\end{proof}

To make use of the above theorem to describe the affine variety $V(I)$,
one first finds a Gr\"obner basis $\{G_1,\ldots,G_s\}$ for $I$
with respect to the lex monomial order.  Theorem \ref{grobnerprop}(b)
implies that we can partition the $G_i$'s so that the first set are polynomials
in the variables $\{z_{i_1,j_1},\ldots,z_{i_t,j_t}\}$, the second set are polynomials
in the variables $\{z_{i_2,j_2},\ldots,z_{i_t,j_t}\}$, and so on, i.e.,
the number of variables is eliminated as you move through the partitions.
In some (but not all) cases, one or more of the $G_i$'s may
only contain one variable.  We can then find roots of these polynomials (either explicitly
or numerically),  and then using
these solutions, find roots to the other polynomials.  

We illustrate how to use Gr\"obner bases to eliminate some znz-patterns $\A$
as being potentially nilpotent over $\F$.  We will study the following
pattern in more detail in the next section.

\begin{example} \label{converse}
We consider the znz-pattern
\[\A = \begin{bmatrix} * & * & 0 \\
0 & * & * \\
* & 0 & *
\end{bmatrix}\]
and let $\F = \mathbb{R}$.  In this case,
the generators of the ideal $I_{\A}$ are 
\[I_{\A} = (z_{1,1} + z_{2,2} + z_{3,3}, ~~z_{1,1}z_{2,2} + z_{1,1}z_{3,3} + z_{2,2}z_{3,3},
~~ z_{1,2}z_{2,3}z_{3,1} + z_{,1}z_{2,2}z_{3,3}).\]
We can use a computer algebra program
to check that $I_{\A}:(z_{1,1}z_{1,2}z_{2,2}z_{2,3}z_{3,1}z_{3,3})^{\infty} \neq (1)$.  
Thus Theorem \ref{saturation} does not tell us if  
$\A$ is not potentially nilpotent
over $\R$.

We use either CoCoA or Macaulay 2 to find a Gr\"obner basis for $I_{\A}$:
\[\{z_{1,1} + z_{2,2} + z_{3,3},~~z_{1,2}z_{2,3}z_{3,1} + z_{3,3}^3,
~~z_{2,2}^2 + z_{2,2}z_{3,3} + z_{3,3}^2\}.\]
Notice that the last polynomial contains the fewest number of variables.
If $\A$ was potentially nilpotent, then there exists 
$\underline{a} =(a_{1,1},a_{1,2},a_{2,2},a_{2,3},a_{3,1},a_{3,3}) \in \mathbb{R}^6$ such that 
$M_{\A}(\underline{a})$ is nilpotent,
and specifically, $\underline{a}$ is a zero of all three polynomials in the Gr\"obner basis.  
Note $a_{3,3}$ must be a nonzero real number.  But for any nonzero 
real number $a \in \mathbb{R}$, the last polynomial from the Gr\"obner basis 
implies that $a_{2,2}$ will then have to satisfy
\[z_{2,2}^2 + az_{2,2} + a^2 = 0 \Leftrightarrow z_{2,2} = a\left(\frac{-1 \pm \sqrt{-3}}{2}\right).\]
But then for every nonzero choice of $a \in \mathbb{R}$,  $a_{2,2} \not\in \R$.
Hence, $\A$ is not potentially nilpotent over $\mathbb{R}$.
Observe that this example shows that the converse of Theorem \ref{saturation}(b)
is false.
\end{example}


\section{Graphs without $k$-cycles with $k$ small: a necessary condition}\label{condition}

Let $D(\A)$ be the digraph associated to the znz-pattern $\A$.
It is known that if $\A$ is potentially nilpotent over $\F = \R$,
and if $D(\A)$ has at least two loops, then $D(\A)$ must
have a 2-cycle.
See, for example, \cite[Lemma 3.2]{DHHMMPSV} which considers
the signed case, but the proof also holds in the non-signed case.
When $\F \neq \R$, then
this necessary condition need not hold, as shown in the example
below:

\begin{example}  \label{z7example}
Let $\A$ be the pattern of Example \ref{converse}.
The associated graph $D(\A)$ has three loops but no two cycles, 
so \cite[Lemma 3.2]{DHHMMPSV} implies that $\A$ is not potentially
nilpotent over $\R$.  However, $\A$ is potentially nilpotent
over $\F = \Z_7$ as demonstrated with the realization
\[\begin{bmatrix}
4_\F & 1_\F & 0 \\
0 & 2_\F & 1_\F \\
-1_\F & 0 & 1_\F
\end{bmatrix}.\]
\end{example}

Our goal in this section is to understand and generalize this example.
More precisely,
we provide a necessary condition on $\F$ for a znz-pattern $\A$
to be potentially nilpotent over $\F$ if $D(\A)$ has loops,
but no $k$-cycles of small size.
We begin by recalling the definition of the roots of unity and one
result concerning these numbers.

\begin{definition}  Fix a field $\F$.  We say that $\F$ {\it contains
all the $m^{th}$ roots of unity} if all of
the $m$ roots of the polynomial $x^m-1_\F = (x-1_\F)(x^{m-1} + x^{m-2} + \cdots + x + 1_\F)$
belong to $\F$, that is, $x^m-1_\F$ factors into $m$ linear forms over $\F$.
\end{definition}

\begin{lemma}  \label{rootsofunitylemma}
Fix a field $\F$ and integer $m \geq 2$.  Suppose that 
there is a solution $(a_1,\ldots,a_m) \in \F^m$ to the $m-1$ 
elementary symmetric polynomial equations
\begin{eqnarray*}
z_1 + z_2 + \cdots + z_m & = & 0 \\
z_1z_2 + \cdots +z_{m-1}z_m & = & 0 \\
& \vdots &  \\
z_1z_2\cdots z_{m-1} + \cdots +z_2z_3\cdots z_m &= & 0
\end{eqnarray*}
with all $a_j \neq 0$.
Then $\F$ contains all the $m^{th}$ roots of unity.
\end{lemma}

\begin{proof}
If $(a_1,\ldots,a_m)$ is such a solution, then $(a_1a_m^{-1},\ldots,a_ma_m^{-1})$
is also a solution.  Thus, we can assume $a_m = 1_\F$.  Hence, substituting
$(a_1,\ldots,a_{m-1},1_\F)$ into the above equations and rearranging gives:
\begin{eqnarray*}
a_1 + a_2 + \cdots + a_{m-1} & = & -1_\F \\
a_1a_2 + \cdots +a_{m-2}a_{m-1}  & = & 1_\F \\
& \vdots &  \\
a_1a_2\cdots a_{m-1} &= & (-1_\F)^{m-1}.
\end{eqnarray*}
We claim that $a_1,\ldots,a_{m-1}$
are all the non-identity $m^{th}$ roots of unity.  Indeed,
\footnotesize
\begin{eqnarray*}
(x - a_1)(x-a_2)\cdots (x- a_{m-1})& = & 
x^{m-1} -(a_1+a_2+\cdots a_{m-1})x^{m-2} + (a_1a_2 + \cdots + a_{m-2}a_{m-1})x^{m-3} \\
&&+ \cdots + (-1)^{m-2} (a_1\cdots a_{m-2} + \cdots + a_2\cdots a_{m-1})x + 
(-1)^{m-1}a_1\cdots a_{m-1} \\
& = & x^{m-1} + x^{m-2} + \cdots +x + 1_\F. 
\end{eqnarray*}
\normalsize
That is, $a_1,\ldots,a_{m-1}$ are the zeros of  $x^{m-1} + x^{m-2} + \cdots +x + 1_\F$. 
The conclusion now follows.
\end{proof}

\begin{theorem}  \label{nomcycle}
Let $\A$ be a znz-pattern with $m \geq 2$ nonzero entries
on the diagonal, and suppose that $D(\A)$ has no $k$-cycles
with $2 \leq k \leq m-1$.  If $\A$ is potentially nilpotent
over $\F$, then $\F$ contains all the $m^{th}$ roots of unity.
\end{theorem}

\begin{proof}  After relabelling, we may assume that the nonzero diagonal entries
of $\A$ are at $(1,1),\ldots,(m,m)$.  To simplify notation, let $z_i$ denote
the variable $z_{i,i}$ in the polynomial ring $R_\A$.  Because $D(\A)$
has no $k$-cycles with $2 \leq k \leq m-1$, this implies that the 
first $m-1$ generators of $I_{\A}$ are:
\begin{eqnarray*}
F_1 & = & z_1 + z_2 + \cdots + z_m \\
F_2 & = & z_1z_2 + \cdots +z_{m-1}z_m \\
& \vdots &  \\
F_{m-1} & = & z_1z_2\cdots z_{m-1} + \cdots +z_2z_3\cdots z_m.
\end{eqnarray*}
Let $A \in Q(\A)$ be a realization that is nilpotent.
If $a_{1,1},\ldots,a_{m,m}$ are the nonzero diagonal entries, then
$\underline{a} = (a_{1,1},\ldots,a_{m,m})$ satisfies $F_i(\underline{a}) = 0$
for $i=1,\ldots,m-1$.
Because $a_{j,j} \neq 0$ for $1\leq j\leq m$, Lemma \ref{rootsofunitylemma} implies that
the field $\F$ contains all the $m^{th}$ roots of unity.
\end{proof}

\begin{corollary}\label{notpotnil}
Let $\A$ be a znz-pattern with $m \geq 2$ nonzero entries
on the diagonal, and suppose that $D(\A)$ has no $k$-cycles
with $2 \leq k \leq m$.  Then $\A$ is not potentially 
nilpotent over any $\F$. 
\end{corollary}

\begin{proof}
We  use the notation of the proof of Theorem \ref{nomcycle}. 
Because
$\A$ has no $k$-cycle with $2 \leq k \leq m$, we  have
$F_m = z_1z_2 \cdots z_m \in I_{\A}$.  Now apply Lemma \ref{monomiallemma}.
\end{proof}

Using the above theorem, we can give a infinite family $\A_n$ below
of potentially nilpotent znz-patterns. In \cite{Britz}, this family 
was demonstrated to fail to be potentially nilpotent for $\F=\R$.

\begin{theorem}\label{no2cycle}
Fix a field $\F$, and for each $n\geq 3$, let $\A_n$ denote the $n \times n$ znz-pattern
\[\A_n =
\begin{bmatrix} 
*        & *         &  0    &\cdots &\cdots&0\\
0        & *         & *    &  \ddots       & &\vdots \\
\vdots&  \ddots&\ddots       & \ddots       &\ddots&\vdots  \\
 \vdots         &            &      & \ddots         & * &0\\
0        &  0        &      &  \ddots  &*& *\\
*        & 0         &  0  &  \cdots  & 0 & * 
\end{bmatrix} 
\]
Then $\A_n$ is potentially nilpotent over $\F$ if and only if $\F$ contains
all the $n^{th}$ roots of unity.
\end{theorem}

\begin{proof}
The graph of $D(\A)$ is an $n$-cycle with a loop at each vertex.  Thus,
one direction follows immediately from Theorem \ref{nomcycle}
since $D(\A)$ has $n$ loops and no $k$-cycles for $2 \leq k \leq n-1$.  For
the converse direction, suppose that $\F$ contains all the $n^{th}$ roots
of unity.  Let $\zeta_1,\ldots,\zeta_{n-1},1_\F$ be the $n$ distinct
$n^{th}$ roots of unity.
Then the matrix
\[
A_n =
\begin{bmatrix} 
\zeta_1 & 1_\F & 0 &\cdots & \cdots&0\\
0 & \zeta_2 & 1_\F &  0      &  & 0\\
\vdots & \ddots &\ddots     &    \ddots       &\ddots &\vdots \\
  \vdots  &                 &    &\ddots                 &\ddots                  &0\\
0 &  0      &   &\ddots    &\zeta_{n-1}& 1_\F\\
-1_\F & 0   &  0  &  \cdots  & 0 & 1_\F 
\end{bmatrix}\]
is a desired realization.
\end{proof}

\begin{corollary}  Fix a prime $p$.
If $p \equiv 1 \pmod{n}$,
then $\A_n$ is potentially nilpotent over $\F = \Z_p$.
\end{corollary}

\begin{proof}  When $p \equiv 1 \pmod{n}$, then by \cite[Theorem 2.4]{LN},
the field $\Z_p$ contains all the $n^{th}$ roots of unity.  Now apply
the Theorem~\ref{no2cycle}.
\end{proof}

\begin{example}
Theorem \ref{no2cycle} gives a new way to explain why the pattern $\A = \A_3$
of Example \ref{converse} is not potentially nilpotent over $\R$.  
Because $D(\A)$ has three loops, but no two cycles,  if $\A$ were potentially
nilpotent over $\F$, then $\F$ must contain all the $3^{rd}$ roots of unity.
However, $\R$ does not have this property.  However, when $\F = \Z_7$,
all three roots of unity belong to $\F$.  This is the reason
why we can find a realization in Example \ref{z7example}.
\end{example}


\section{Potentially nilpotent matrices of small order over finite fields}\label{small}

In this section, we employ the tools of previous sections
to classify all znz-patterns $\A$ that are potentially
nilpotent over $\F$ of order two or three when $\F = \Z_p$,
with $p$ a prime number.  As a consequence of Lemma \ref{irreduciblereduction}, it
suffices to classify all znz-patterns of order two or three that are
irreducible.

We begin with the $2 \times 2$ case by showing a much stronger result:

\begin{theorem} \label{2x2case}
Let $\F$ be any field.  Then the znz-pattern
\[\A = \begin{bmatrix} * & * \\
* & * 
\end{bmatrix}\]
is the only irreducible $2 \times 2$ potentially nilpotent pattern over $\F$.
\end{theorem}

\begin{proof}
The only irreducible $2 \times 2$ znz-patterns are 
 \[\begin{bmatrix} 0 & * \\
* & 0
\end{bmatrix}, \
\begin{bmatrix} * & * \\
* & 0
\end{bmatrix},  \
\begin{bmatrix} 0 & * \\
* & *
\end{bmatrix}, \ \mbox{\rm{and}} \
\begin{bmatrix} * & * \\
* & * 
\end{bmatrix}.
\]
The first three patterns patterns cannot be potentially nilpotent over $\F$
by Corollary \ref{corsaturation}.  The matrix
\[\begin{bmatrix} 1_{\F} & 1_{\F} \\
-1_{\F} & -1_{\F} 
\end{bmatrix}.
\]
is a desired realization of $\A$.
\end{proof}

The following lemma 
is used to shorten some of the cases
in the next theorem.

\begin{lemma}  \label{specialcases}
Let $\A$ be an irreducible $n \times n$ znz-pattern.
Let $D(\A)$ be the associated digraph.
\begin{enumerate}
\item[(a)] If $\F = \Z_2$ and $D(\A)$ has an odd number of  loops,
then $\A$ is not potentially nilpotent over $\F$.
\item[(b)] If $\F = \Z_2$ and $D(\A)$ has exactly two loops
and two $2$-cycles, then $\A$ is not potentially nilpotent
over $\F=\Z_2$.
\item[(c)] The only solutions to the equation $x+y+z = 0$ with $x,y,z \in \Z_3$
and $x$, $y$, $z$ nonzero are $(1_\F,1_\F,1_\F)$ and $(2_\F,2_\F,2_\F)$.
\end{enumerate}
\end{lemma}

\begin{proof} $(a)$ Suppose that $D(\A)$ has loops at 
$(i_1,i_1),\ldots,(i_m,i_m)$ with $m =2k+1$.  Then 
$z_{i_1,i_1} + \cdots + z_{i_m,i_m}  \in I_{\A}$.  When $\F = \Z_2$,
we must have $z_{i,j} = 1_{\F}$ for all $(i,j)$.  But this
would imply that $1_\F + \cdots + 1_\F = m_\F = 0$,
which is false.

$(b)$  Suppose that the diagonal entries of $\A$ are at $(i_1,i_1)$ and $(i_2,i_2)$ 
and the two 2-cycles are $(i_3,j_3)$ and $(i_4,j_4)$.  Then
the polynomial 
\[F_2=z_{i_1,i_1}z_{i_2,i_2}+ z_{i_3,j_3}z_{j_3,i_3} + z_{i_4,j_4}z_{j_4,i_4}  \in I_{\A}.\]
If $\F = \Z_2$, then the only nonzero choice for $z_{i,j}$ is $1_{\F}$.
It then follows that $F_2$ cannot equal zero in $\F = \Z_2$.

$(c)$ This statement follows from inspection.
\end{proof}

We say that two patterns are \emph{equivalent} if they have the same digraph.

\begin{theorem} \label{3x3case}
Fix a prime $p$ and an irreducible $3 \times 3$ znz-pattern $\A$.
Then $\A$ is potentially nilpotent over $\F = \Z_p$ if and only if, up to equivalence, 
$\A$ and $p$ have one of the following forms:
\noindent\begin{enumerate}
\item[1.] $\A = \begin{bmatrix} 0 & * & 0 \\
* & 0 & * \\
0 & * & 0 
\end{bmatrix},
 \begin{bmatrix} * & * & 0 \\
* & 0 & * \\
* & 0 & *
\end{bmatrix}$,
$\begin{bmatrix} 0 & * & 0 \\
* & * & * \\
* & 0 & * 
\end{bmatrix},$  or
 $\begin{bmatrix} * & * & * \\
* & * & * \\
* & * & 0 
\end{bmatrix}$,
and $p$ arbitrary.
\item[2.] $\A = \begin{bmatrix} * & * & 0 \\
* & 0 & * \\
0 & * & *
\end{bmatrix},
 \begin{bmatrix} * & * & * \\
* & * & * \\
* & 0 & 0 
\end{bmatrix},
\begin{bmatrix} * & * & * \\
* & 0 & * \\
* & 0 & *
\end{bmatrix},$ 
$\begin{bmatrix} * & * & * \\
* & * & * \\
* & 0 & * 
\end{bmatrix},$
$\begin{bmatrix} 0 & * & * \\
* & 0 & * \\
* & * & 0 
\end{bmatrix}$, or
$\begin{bmatrix} * & * & * \\
* & * & * \\
* & * & * 
\end{bmatrix}$,
and $p \neq 2$.
\item[3.]
 $ \A= \begin{bmatrix} * & * & 0 \\
* & * & * \\
0 & * & *
\end{bmatrix}$,
$\begin{bmatrix} * & * & 0 \\
* & * & * \\
* & 0 & * 
\end{bmatrix}$, or
$\begin{bmatrix} 0 & * & * \\
* & * & * \\
* & 0 & * 
\end{bmatrix}$,
and $p \neq 2$ or $3$.
\item[4.]  $\A = \begin{bmatrix} * & * & 0 \\
0 & * & * \\
* & 0 & * 
\end{bmatrix}$ and the three roots of $(x^3-1_\F)$ are elements of $\F = \mathbb{Z}_p$.
\end{enumerate}
\end{theorem}

\begin{proof}
We do a case-by-case analysis, by considering all $3 \times 3$ irreducible znz-patterns
$\A$.  Recall that a pattern $\A$ is irreducible if and only if 
the digraph $D(\A)$ is strongly connected.  We break our proof into five cases,
where each case corresponds to one of the five non-isomorphic graphs on three vertices that
is strongly connected and with no loops.  Each case is then broke into
sub-cases, where each sub-case considers the locations of the loops.

\noindent
{\bf Case 1.}  The non-loop edges are $(1,2),(2,3)$, and $(3,1)$.
\vspace{.25cm}

\noindent
In this case, we need to consider four znz-patterns:
\[
\A_{1,1} = \begin{bmatrix} 0 & * & 0 \\
0 & 0 & * \\
* & 0 & 0 
\end{bmatrix} \
\A_{1,2} = \begin{bmatrix} * & * & 0 \\
0 & 0 & * \\
* & 0 & 0 
\end{bmatrix} \
\A_{1,3} = \begin{bmatrix} * & * & 0 \\
0 & * & * \\
* & 0 & 0 
\end{bmatrix}\ 
\A_{1,4} = \begin{bmatrix} * & * & 0 \\
0 & * & * \\
* & 0 & * 
\end{bmatrix}.\]
Patterns $\A_{1,1},\A_{1,2},$ and $\A_{1,3}$ cannot be potentially nilpotent
over any field $\F$ by Corollary \ref{corsaturation} since one of
the generators of $I_{A_{1,j}}$ for $j=1,2,3$ will be a monomial.
On the other hand, $\A_{1,4}$ is potentially nilpotent over $\F$ if and only if
the polynomial $x^3-1_\F$ factors into linear forms in $\F = \Z_p$.
This is a special case of Theorem \ref{no2cycle}.

\noindent
{\bf Case 2:} The non-loop edges are $(1,2),(2,1),(2,3)$, and $(3,2)$.
\vspace{.25cm}

\noindent
In this case, we need to consider six znz-patterns:
\[
\A_{2,1} = \begin{bmatrix} 0 & * & 0 \\
* & 0 & * \\
0 & * & 0 
\end{bmatrix}\ 
\A_{2,2} = \begin{bmatrix} * & * & 0 \\
* & 0 & * \\
0 & * & 0 
\end{bmatrix} \
\A_{2,3} = \begin{bmatrix} 0 & * & 0 \\
* & * & * \\
0 & * & 0 
\end{bmatrix}\]\[
\A_{2,4} = \begin{bmatrix} * & * & 0 \\
* & * & * \\
0 & * & 0 
\end{bmatrix} \
\A_{2,5} = \begin{bmatrix} * & * & 0 \\
* & 0 & * \\
0 & * & *
\end{bmatrix}\ 
 \A_{2,6} = \begin{bmatrix} * & * & 0 \\
* & * & * \\
0 & * & *
\end{bmatrix}.\]
We can eliminate patterns $\A_{2,2},\A_{2,3}$ and $\A_{2,4}$ as being
potentially nilpotent over any $\F$ by using Corollary
\ref{corsaturation}.  Pattern $\A_{2,1}$ is potentially
nilpotent over any $\F$ since $1_\F,-1_\F \in \F$ and 
\[ \begin{bmatrix} 0 & 1_{\F} & 0 \\
1_\F & 0 & 1_\F \\
0 & -1_\F & 0 
\end{bmatrix}\] 
is a desired realization.

Using Theorem \ref{saturation} and CoCoA, we can show that the 
pattern\footnote[3]{The pattern $\A_{2,5}$ is the antipodal
tridiagonal pattern $T_3$  studied in \cite{DJOvdD}.} $\A_{2,5}$
is not potentially nilpotent over $\F = \Z_2$
(or alternatively, see Examples \ref{examplenotation} and \ref{examplesaturation}).  If $p \neq 2$,
then $\A$ is potentially nilpotent over $\F = \Z_p$ since
\[ \begin{bmatrix} 1_{\F} & -2^{-1}_{\F} & 0 \\
1_\F & 0 & 2^{-1}_\F \\
0 & -1_\F & -1_{\F} 
\end{bmatrix}\] 
is a desired realization.  

The pattern $\A_{2,6}$ is not potentially nilpotent over 
$\F = \Z_2$ by Lemma \ref{specialcases} (a).
Also,
$\A_{2,6}$ is not potentially nilpotent over $\F = \Z_3$.  We can show this by
calculating the Gr\"obner basis of $I_{\A_{2,6}}$ in the ring
$R = \Z_{3}[z_{1,1},z_{1,2},z_{2,1},z_{2,2},z_{2,3},z_{3,2},z_{3,3}]$:
\[\{z_{1,1} + z_{2,2} + z_{3,3},\hspace{.2cm} 
z_{1,2}z_{2,1} + z_{2,2}^2 + z_{2,3}z_{3,2} + z_{2,2}z_{3,3} + z_{3,3}^2,\hspace{.2cm} 
z_{2,2}z_{2,3}z_{3,2} - z_{2,3}z_{3,2}z_{3,3} + z_{3,3}^3\}.\]
Since $z_{1,1}+z_{2,2}+z_{3,3} = 0$, by Lemma \ref{specialcases} (c),
we need only consider the cases that $z_{1,1}=z_{2,2} = z_{3,3} = 1_\F$
or they all equal $2_\F$.
In either case, solving for nonzero roots of the last polynomial of the Gr\"obner basis we get
$z_{2,3}z_{3,2}-z_{2,3}z_{3,2} + 1_\F = 0$ or $2_\F z_{2,3}z_{3,2}-2_\F z_{2,3}z_{3,2} + 2_\F = 0$,
neither of which has a solution in $\Z_3$.  So, this pattern is not
potentially nilpotent of $\Z_3$.  On the other hand,
if $p \neq 3$,
then
\[ \begin{bmatrix} 2_\F & 2_\F & 0 \\
-4_\F & -3_\F & 1_\F \\
0 & 1_\F & 1_\F
\end{bmatrix}\]
is a desired realization. 

\noindent
{\bf Case 3.}  The non-loop edges are $(1,2),(2,1),(2,3)$, and $(3,1)$.
\vspace{.25cm}

\noindent
We now need to consider eight znz-patterns:
\[
\A_{3,1} = \begin{bmatrix} 0 & * & 0 \\
* & 0 & * \\
* & 0 & 0 
\end{bmatrix}\ 
\A_{3,2} = \begin{bmatrix} * & * & 0 \\
* & 0 & * \\
* & 0 & 0 
\end{bmatrix}\ 
\A_{3,3} = \begin{bmatrix} 0 & * & 0 \\
* & * & * \\
* & 0 & 0 
\end{bmatrix} \
\A_{3,4} = \begin{bmatrix} 0 & * & 0 \\
* & 0 & * \\
* & 0 & * 
\end{bmatrix}\]\[
\A_{3,5} = \begin{bmatrix} * & * & 0 \\
* & * & * \\
* & 0 & 0 
\end{bmatrix} \
\A_{3,6} = \begin{bmatrix} * & * & 0 \\
* & 0 & * \\
* & 0 & *
\end{bmatrix} \
\A_{3,7} = \begin{bmatrix} 0 & * & 0 \\
* & * & * \\
* & 0 & * 
\end{bmatrix} \
\A_{3,8} = \begin{bmatrix} * & * & 0 \\
* & * & * \\
* & 0 & * 
\end{bmatrix}.
\]
Matrices $\A_{3,j}$ for $j=1,\ldots,5$ are not potentially nilpotent
over any field by Corollary \ref{corsaturation}.
The matrices $\A_{3,6}$  and $\A_{3,7}$ are potentially nilpotent over any field $\F$
with realizations
\[\begin{bmatrix} -1_\F & -1_\F & 0 \\
1_\F & 0 & 1_\F \\
1_\F & 0 & 1_\F
\end{bmatrix}~~\mbox{and}~~
\begin{bmatrix} 0 & -1_\F & 0 \\
1_\F & -1_\F & 1_\F \\
1_\F & 0 & 1_\F
\end{bmatrix}
\]
respectively.  Finally, the pattern $\A_{3,8}$ cannot be potentially
nilpotent over $\Z_2$ by Lemma \ref{specialcases} (a).  Also,
this pattern is not nilpotent over $\Z_3$;  again, we use
a Gr\"obner basis of $I_{\A_{3,8}}$:
\[\{z_{1,1} + z_{2,2} + z_{3,3}, \hspace{.2cm}z_{1,2}z_{2,3} + z_{2,2}^2 + z_{2,2}z_{3,3} + z_{3,3}^2,
\hspace{.2cm} 
z_{1,2}z_{2,3}z_{3,1} + z_{3,3}^3, \]\[
z_{2,2}^2z_{2,3}z_{3,1} + z_{2,2}z_{2,3}z_{3,1}z_{3,3} + z_{2,3}z_{3,1}z_{3,3}^2 - z_{2,1}z_{3,3}^3\}.\]
By Lemma \ref{specialcases} (c), the first polynomial implies $z_{1,1} = z_{2,2} = z_{3,3} = 1_\F$
or $2_\F$.  In the first case, the second polynomial
reduces to $z_{1,2}z_{2,3} + 1_\F + 1_\F + 1_\F = z_{1,2}z_{2,3} $ which is  nonzero
 in $\Z_3$.  Similarly, in the second case, the second polynomial equation
becomes $z_{1,2}z_{2,3} + 4_\F + 4_\F + 4_\F = z_{1,2}z_{2,3} \neq 0$. 
 If $p \neq 2,3$,
this pattern is potentially nilpotent
with realization
\[\begin{bmatrix} -2_\F & -1_\F & 0 \\
3_\F & 1_\F & 1_\F \\
1_\F & 0 & 1_\F
\end{bmatrix}.
\]

\noindent
{\bf Case 4.}  The non-loop edges are $(1,2),(2,1),(2,3),(1,3)$, and $(3,1)$.
\vspace{.25cm}

\noindent
We need to consider eight znz-patterns:
\[
\A_{4,1} = \begin{bmatrix} 0 & * & * \\
* & 0 & * \\
* & 0 & 0 
\end{bmatrix}\ 
\A_{4,2} = \begin{bmatrix} * & * & * \\
* & 0 & * \\
* & 0 & 0 
\end{bmatrix}\ 
\A_{4,3} = \begin{bmatrix} 0 & * & * \\
* & * & * \\
* & 0 & 0 
\end{bmatrix} \ 
\A_{4,4} = \begin{bmatrix} 0 & * & * \\
* & 0 & * \\
* & 0 & * 
\end{bmatrix}\]\[
\A_{4,5} = \begin{bmatrix} * & * & * \\
* & * & * \\
* & 0 & 0 
\end{bmatrix}\ 
\A_{4,6} = \begin{bmatrix} * & * & * \\
* & 0 & * \\
* & 0 & *
\end{bmatrix}\ 
\A_{4,7} = \begin{bmatrix} 0 & * & * \\
* & * & * \\
* & 0 & * 
\end{bmatrix} \
\A_{4,8} = \begin{bmatrix} * & * & * \\
* & * & * \\
* & 0 & * 
\end{bmatrix}.
\]
We can use Corollary \ref{corsaturation} to eliminate the znz-patterns
$\A_{4,j}$ for $j=1,\ldots,4$.  The patterns $\A_{4,j}$ for $j=5,\ldots,8$
fail to be potentially nilpotent over $\F=\Z_2$ by Lemma \ref{specialcases}.
Indeed, the first three are eliminated by part $(b)$, while the last
is eliminated by $(a)$.

The matrices $\A_{4,5}$ and $\A_{4,6}$ are potentially nilpotent
over any $\F = \Z_p$ with $p \neq 2$ with realizations
\[ 
\begin{bmatrix} 1_\F & 1_\F & 1_\F \\
1_\F & -1_\F & -1_\F \\
2_\F & 0 & 0 
\end{bmatrix}
~~\mbox{and}~~
\begin{bmatrix} 1_\F & 1_\F & 1_\F \\
1_\F & 0 & 2_\F^{-1} \\
-2_\F & 0 & -1_\F
\end{bmatrix}
~~\mbox{respectively.}\]

When $p=3$, two of the polynomials in the Gr\"obner basis of $I_{\A_{4,7}}$
are $z_{1,2}z_{2,1} + z_{3,1}z_{1,3} + z_{3,3}^2$ and $z_{1,2}z_{2,3}z_{3,1} - z_{1,3}z_{3,1}z_{3,3} + z_{3,3}^3$.
By Lemma \ref{specialcases} (c), the
first polynomial can only equal zero if $z_{1,2}z_{2,1} = z_{3,1}z_{1,3} = z_{3,3}^2$ in $\Z_3$.
Since $z_{3,3} \neq 0$, we will always have $z_{3,3}^2 = 1_\F$ in $\Z_3$.
Hence $z_{1,3}z_{3,1}=1$ and the second polynomial reduces to $z_{1,2}z_{2,3}z_{3,1}\neq 0$
So, $\A_{4,7}$
is not potentially nilpotent over $\Z_3$.  However, when $p \neq 3$,
we have the realization
\[\begin{bmatrix}
0 & -2_\F & 1_\F \\
1_\F & -1_\F & 3_{\F}\cdot 2_{\F}^{-1} \\
1_\F & 0 & 1_\F
\end{bmatrix}.
\]

The last pattern $\A_{4,8}$ is potentially nilpotent over $\Z_p$ for any 
prime $p \geq 3$ with realization:
\[\begin{bmatrix}
-2_{\F}& -4_\F & 1_\F \\
1_\F & 1_\F & 4_{\F}^{-1} \\
1_\F & 0 & 1_\F
\end{bmatrix}.
\]
\noindent
{\bf Case 5.}  The non-loop edges are $(1,2),(2,1),(3,2),(2,3),(1,3)$, and $(3,1)$.
\vspace{.25cm}

\noindent
We now need to consider the remaining four irreducible znz-patterns:
\[
\A_{5,1} = \begin{bmatrix} 0 & * & * \\
* & 0 & * \\
* & * & 0 
\end{bmatrix}\ 
\A_{5,2} = \begin{bmatrix} * & * & * \\
* & 0 & * \\
* & * & 0 
\end{bmatrix} \ 
\A_{5,3} = \begin{bmatrix} * & * & * \\
* & * & * \\
* & * & 0 
\end{bmatrix} \
\A_{5,4} = \begin{bmatrix} * & * & * \\
* & * & * \\
* & * & * 
\end{bmatrix}.\]
Pattern $\A_{5,2}$ is not potentially nilpotent over any $\F$ by Corollary
\ref{corsaturation}.  Also, by Lemma \ref{specialcases} (a), the pattern
$\A_{5,4}$ is not potentially nilpotent over $\F = \Z_2$.

For the pattern $\A_{5,1}$, we have $z_{1,2}z_{2,1} + z_{1,3}z_{1,3}+z_{2,3}z_{3,2} \in I_{\A_{5,1}}$.
This has no nonzero solution in $\Z_2$. 
When $p \neq 2$, one can use the realization:
\[\begin{bmatrix}
0 & -1_\F & 1_\F \\
4_\F & 0 & 2_\F \\
2_\F & 1_\F & 0 
\end{bmatrix}.
\]
The pattern $\A_{5,3}$ is potentially nilpotent over any $\F = \Z_p$ with
realization:
\[\begin{bmatrix}
1_\F & 1_\F & 1_\F \\
-1_\F & -1_\F & -1_\F \\
1_\F & 1_\F & 0 
\end{bmatrix}.
\]
Finally, for any $p \geq 3$, the pattern $\A_{5,4}$ is potentially nilpotent over $\F = \Z_p$;
indeed, one such realization is
\[\begin{bmatrix}
1_\F & 1_\F & 1_\F \\
1_\F & 1_\F & 1_\F \\
-2_\F & -2_\F & -2_\F 
\end{bmatrix}.
\]
\end{proof}

\begin{remark}  We point out three interesting
facts that arise from this classification.  
First, all the irreducible patterns that are not potentially
nilpotent over any $\Z_p$ are in fact not potentially nilpotent
over any field $\F$.  As a consequence, to determine which
irreducible patterns are potentially nilpotent over $\F = \R$,
it suffices to consider only the irreducible patterns
that appear in the statement of Theorem \ref{3x3case}.  Moreover,
the realizations given in the proof of Theorem \ref{3x3case} show
that all of these irreducible patterns are potentially nilpotent over $\R$
{\it except} the pattern in Case $4$. 
Second, if $\A$ is potentially nilpotent
over $\Z_p$, it does not necessarily follow that any superpattern of $\A$
continues to be potentially nilpotent over $\Z_p$.    And third, notice that {\it none}
of the cases in Case 4 are potentially nilpotent over $\Z_2$.    This lends itself to a natural question:  what digraphs $D(\A)$ have the property that
 $\A$ fails to be potentially nilpotent
over some field $\F$, regardless of the placement of the loops?
\end{remark}


\end{document}